\documentclass[12pt]{amsart}

\newcommand{\sech}{\operatorname{sech}}
\newcommand{\defeq}{\stackrel{\rm{def}}{=}}

\usepackage{amssymb}
\usepackage{amsthm}
\usepackage{amsxtra}

\usepackage{fancyhdr}
\usepackage{nccmath}
\usepackage{graphicx}
\usepackage{color}
\usepackage{amsmath}
\usepackage{listings} 
\usepackage{float}
\usepackage{cases} 
\usepackage{threeparttable}
\usepackage{dcolumn}
\usepackage{multirow}
\usepackage{booktabs}

\usepackage{caption}
\usepackage{subcaption}

\newtheorem{theorem}{Theorem}[section]

\newtheorem{proposition}{Proposition}[section]

\newtheorem{corollary}[proposition]{Corollary}

\theoremstyle{remark}
\newtheorem{remark}[proposition]{Remark}

\numberwithin{equation}{section}

\ifx\pdfoutput\undefined
  \DeclareGraphicsExtensions{.pstex, .eps}
\else
  \ifx\pdfoutput\relax
    \DeclareGraphicsExtensions{.pstex, .eps}
  \else
    \ifnum\pdfoutput>0
      \DeclareGraphicsExtensions{.pdf}
    \else
      \DeclareGraphicsExtensions{.pstex, .eps}
    \fi
  \fi
\fi

\title[Conservative gKdV scheme]{Arbitrarily high-order conservative schemes for the generalized Korteweg-de Vries equation}

\author{Kai Yang}
\date{\today}
\address{Florida International University}

\begin{document}
\maketitle
\begin{abstract}
This paper proposes a new class of arbitrarily high-order conservative numerical schemes for the generalized Korteweg-de Vries (KdV) equation. This approach is based on the scalar auxiliary variable (SAV) method. The equation is reformulated into an equivalent system by introducing a scalar auxiliary variable, and the energy is reformulated into a sum of two quadratic terms. Therefore, the quadratic preserving Runge-Kutta method will preserve { all the three invariants (momentum, mass and the reformulated energy) in the discrete time flow (assuming the spatial variable is continuous).}
 With the Fourier pseudo-spectral spatial discretization, the scheme conserves the first and third invariant quantities (momentum and energy) exactly in the { space-time full} discrete sense. The discrete mass possesses the precision of the spectral accuracy. { Our numerical experiment shows the great efficiency of this scheme in simulating the breathers for the mKdV equation.}
\end{abstract}


\section{Introduction}
This paper considers the numerical approximation of the generalized Korteweg–de Vries (gKdV) equation
\begin{align}\label{E:gKdV}
\begin{cases}
u_t=-(u_{xx}+\frac{1}{p}u^p)_x, \quad x \in \mathbb{R}, \,\, t>0 \\
u(x,0)=u_0,
\end{cases} 
\end{align}
where $p$ is a positive integer.  

When $p=2$, it is known as the KdV equation.
When $p=3$, it is the modified KdV (mKdV) equation. One appealing feature of these two cases is that they form the integrable systems and possess infinitely many integrals, which are invariant in time, see \cite{MGK1968}. This is accounted by reformulating the KdV equation into the Lax pairs \cite{Lax1968}.
When $p \geq 4$, it is known as the generalized KdV equation, which is not integrable. In general,
the equation \eqref{E:gKdV} conserves the following three invariant quantities:
\begin{align}
&I[u(t)] \defeq \int u(x,t) dx =I[u_0];   \label{E:momentum}\\
&M[u(t)] \defeq \int [u(x,t)]^2 dx =M[u_0];  \label{E:mass} \\
&E[u(t)] \defeq \int \left[ \frac{1}{2} \left( u_x(x,t)\right )^2-\frac{1}{p(p+1)} \left(u(x,t)\right)^{p+1} \right] dx = E[u_0], \label{E:energy}
\end{align} 
known as the momentum, mass and energy (Hamiltonian), respectively.

These evolution equations are among the simplest of a general class of models featuring nonlinear advection (the term $(\frac{1}{p}u^p)_x$ in this case). This family of equations arise as mathematical models for the propagation of physical waves in a wide variety of situations, such as shallow-water waves with weakly non-linear restoring forces, long internal waves in a density-stratified ocean, ion-acoustic waves in a plasma, acoustic waves on a crystal lattice, etc., e.g., see \cite{BBM1972},  \cite{JK1972}, \cite{Bona1981}, \cite{BPS1981}, \cite{Craig1987}, \cite{BCL2005}. The equations also attracted attention from the mathematical theory side. Analytical and numerical investigations include
well-posedness, \cite{Kato1983}, \cite{KPV1993};
soliton stability, \cite{BSS1987}, \cite{MM2001};
dispersion limit, \cite{GK2012}; 
and
singularity formations, \cite{KP2015}.
Nevertheless, there are still many open questions, and therefore, an efficient and  accurate numerical algorithm would be desirable for future investigations. 

Numerical studies of the gKdV equation trace back to 1970's, ranging from finite difference methods \cite{Goda1977}, \cite{LV2006}, \cite{CM2007}; finite element methods \cite{AW1982}, \cite{BDKM1995}, \cite{KM1990}; spectral methods \cite{FW1978}, \cite{GS2001}, \cite{HS1992}, \cite{MS2000}, \cite{Shen2003}; operator splitting methods \cite{HKR1999}, \cite{HKRT2011} and discontinuous Galerkin methods \cite{CS2008}, \cite{LY2006}, \cite{XS2007},  \cite{YHL2013}, \cite{LY2016}. Besides the order of accuracy and the stiffness from the term $u_{xxx}$, preserving the conservation laws \eqref{E:momentum}-\eqref{E:energy} is also a main concern in designing a numerical method for \eqref{E:gKdV}. Indeed, for conservative PDEs, numerical methods, which can preserve the corresponding invariants, are often advantageous: besides the accuracy, a conservative scheme can preserve good stability properties, especially in long-time simulations. 
On the other hand, to our best knowledge, most of the papers above are only capable of preserving one or, possibly, two quantities from \eqref{E:momentum}-\eqref{E:energy}, with the second order accuracy in time. 

The purpose of this paper is to present a numerical scheme for the gKdV equation that preserves all these three invariant quantities  \eqref{E:momentum}-\eqref{E:energy} exactly in the discrete time, together with an arbitrarily high order of temporal accuracy. 
This is achieved by applying the scalar auxiliary variable (SAV) approach. { The SAV approach, which was developed from the IEQ approach (\cite{YANG2016} and \cite{YZW2017}), was proposed for minimizing the free energy by gradient flows, aiming to keep the energy stable in the discrete time flow, see \cite{SXY2018}, \cite{SX2018} and convergence analysis in \cite{SXY2019}. 
The main difference between these two methods is that while the IEQ method is introducing an additional variable function from the nonlinear terms in the equation, the SAV method is introducing a scalar function from the potential part in the energy. Hence, when considering the spatial discretized system (assume $u$ is discretized into $N$ grid points or nodes), the IEQ method results in a $2N \times 1$ system, and the SAV approach results in an $(N+1) \times 1$ system, which reduces the computational cost significantly.
}
The SAV or IEQ methods are then applied to the Hamiltonian PDEs in \cite{CS2020} and  \cite{CWJ2021}. Inspired by the same idea, we reformulate the equation \eqref{E:gKdV} into an equivalent system by introducing an auxiliary variable. The reformulated system conserves the original momentum and mass, and also the modified energy, which is rewritten into the sum of two quadratic terms. Therefore, by the standard numerical ODE theory, the quadratic preserving (symplectic) Runge-Kutta methods will preserve all these three invariants exactly in the discrete time flow {(assume the space variable is continuous)}. The standard Fourier pseudo-spectral method is chosen for the spatial discretization. This spatial discretization preserves the momentum and energy exactly in the spatial discrete sense (assuming time is continuous), and consequently, the error of the discrete mass only comes from the spatial discretization, which is of spectral accuracy and usually unnoticeable in our numerical experiments.

This paper is organized as follows. In Section 2, we give the equivalent form of the reformulated gKdV equation \eqref{E:gKdV} and the modified energy \eqref{E:energy} based on the SAV approach. In Section 3, we show that the family of the quadratic preserving (symplectic) Runge-Kutta methods will preserve the momentum, mass and modified energy in the discrete time flow. In Section 4, we describe the spatial discretization.
We prove that the conservation laws \eqref{E:momentum} and \eqref{E:energy} hold in the spatial discrete sense. We also show the error of \eqref{E:mass} from this discretization. Combing with results in Section 3, we prove that the proposed scheme preserves the momentum and energy exactly in the { space-time} fully discretized { scheme in the computation}. In the meanwhile, we also give an error estimate of the discrete mass, which is $\sim C_d \cdot t$ for some constant $C_d$ coming from the error of the spatial discretization. Due to the high accuracy of the Fourier spectral method, the constant $C_d$ is generally on the order of $10^{-14}$, several orders smaller than the tolerance for solving the resulting nonlinear system, and thus, the discrete mass error is almost unnoticeable.
In Section 5, we first describe the numerical algorithm for solving the resulting nonlinear system from the symplectic Runge-Kutta method. { Next, we briefly describe the several existing algorithms for the gKdV equations, such as the modified Crank-Nicholson method, conventional SAV with Leap-Frog (SAV-LF) approach, the Strang-Splitting (SS) method from \cite{HKR1999}, and the 4th order modified exponential time differencing (mETDRK4) method (\cite{CM02} and \cite{KT05}). They will be used in this paper for comparison purpose.}
Finally, we list our numerical examples. Numerical results show the fulfillment of conservation laws and a high accuracy of the proposed scheme, {especially in simulating the breathers, which is a type of highly oscillatory non-dispersing pulse solutions for the mKdV equations.}


\medskip
{\flushleft \bf Acknowledgment:} The author is partially supported by the NSF grant DMS-1927258 (PI: Svetlana Roudenko). The author is thankful for Dr. Roudenko's helpful discussion, reading and remarks on the paper.

\section{Model reformulation by the SAV approach}
In this section, we reformulate the equation \eqref{E:gKdV} into an equivalent system, which possesses a modified energy function of a new variable. 

We first define the inner product for the real valued functions functions $f,g\in L^2(\mathbb{R}):$
\begin{align}\label{D:L2}
(f,g)\defeq \int_{\mathbb{R}} f(x)g(x)dx.
\end{align}
{Assume $u^{p+1} \in L^1_x([0,T])$.} Consider a new scalar variable 
$$v \defeq v(u,t)=\sqrt{(u^p,u)+C_0},$$ 
where $C_0$ is a large enough positive constant to prevent the expression $(u^p,u)+C_0$ under the square root to become negative during our simulation time $t \in [0,T]$. In Section 3, we will provide an adjustment process for $C_0$ when $v(t_m)$ is close to $0$ at some time $t=t_m$. Therefore, for now we only need to choose a constant $C_0$ at $t=t_m$ to make sure $v(t)>0$ for $t\in [t_m,t_{m+1}]$. This can be easily achieved, since we assume that the solution is well-posed (smooth in time) in $t \in [0,T]$.

The equation \eqref{E:gKdV} is then reformulated into the following system 
\begin{align}\label{E:gKdV sav}
\begin{cases}
u_t=-\left( u_{xx}+\frac{1}{p} \frac{u^p v}{\sqrt{(u^p,u)+C_0}} \right)_x \defeq f(u,v),\\
v_t=\frac{p+1}{2\sqrt{(u^p,u)+C_0}} ( u^p ,u_t) \defeq g(u,v),
\end{cases}
\end{align}
with the initial condition 
\begin{align*}
u(x,0)=u_0, \qquad v_0=\sqrt{( u_0^p,u_0)+C_0}.
\end{align*}

Since the modified system \eqref{E:gKdV sav} is identical to \eqref{E:gKdV} { when the spatial and time variables $(x,t)$ are continuous}, the modified system \eqref{E:gKdV sav} conserves the momentum, mass and modified energy in the form
\begin{align}
&I[u(t)] \defeq (u,1)\equiv I[u_0];  \label{E:momentum sav} \\
&M[u(t)] \defeq (u,u) \equiv M[u_0];  \label{E:mass sav} \\
&E[u(t), v(t)] \defeq -\frac{1}{2}(u_{xx},u)-\frac{1}{p(p+1)} (v^2 -C_0) \equiv E[u_0, v_0]. \label{E:energy sav}
\end{align}

\section{Temporal discretization}
In this section, we describe the temporal discretization of the scheme. We first show that the family of the symplectic Runge-Kutta (SRK) methods will conserve all three quantities \eqref{E:momentum sav} \eqref{E:mass sav} and \eqref{E:energy sav} in the discrete time flow. Then, we propose a strategy to adjust the constant $C_0$ in the computation, which makes the reformulated system \eqref{E:gKdV sav} valid all the time and keeps the energy invariant.  

\subsection{Symplectic Runge-Kutta method}
We first recall the $s$-stage collocation Runge-Kutta method. {Consider the IVP of the ODE in general form $${u}_t=f({u},t), \quad {u}(t_0)={u}_0.$$ Define $\tau$ to be the time step.
Let $c_1, \cdots, c_s$ be distinct real numbers (usually $0\leq c_i \leq 1$). The collocation polynomial ${u}(t)$ is a polynomial of degree $s$ satisfying 
$${u}_t(t_0+c_i \tau)=f({u}(t_0+c_i \tau), t_0+c_i \tau),$$
and the numerical solution is defined by ${u}_1={u}(t_0+\tau)$. 

From \cite{GS1969}, we know that the $s$-stage collocation method is equivalent to the $s$-stage Runge-Kutta method. Indeed, consider
$$k_i:={u}_t(t_0+c_i \tau)=f({u}(t_0+c_i \tau),t_0+c_i \tau).$$ 
Letting $l_i(z)=\Pi_{l \neq i} (z-c_l)/(c_i-c_l)$ be the Lagrange polynomial, the Lagrange interpolation formula gives us 
$${u}_t(t_0+z \tau)=\sum_{j=1}^s k_j l_j(z).$$
Integrating from $0$ to $c_i$ for each $i$ yields
$$u(t_0+c_i\tau)=u_0+\tau \sum_{j=1}^s k_j \int_0^{c_i}l_j(z)dz.$$
Denote $a_{ij}=\int_0^{c_i}l_j(z)dz$, $b_i=\int_0^1l_i(z)dz$ and the intermediate values $u_i=u(t_0+c_i\tau)$. We have the Runge-Kutta formulation
\begin{align*}
u_i=u_0+\sum_{j=1}^s k_j a_{ij},
\end{align*}
and, the solution is updated by
\begin{align*}
u(t_0+\tau)=u_0+\sum_{i=1}^s k_i b_i.
\end{align*}
We rewrite the coefficients $\mathbf{A}=(a_{ij})$, $\mathbf{b}=(b_1, b_2,\cdots,b_s)$ and $\mathbf{c}=(c_1, c_2,\cdots,c_s)^T$ in the Butcher's Tableaus (\cite{Butcher1964}):
\[
\begin{array}
{c|c}
\mathbf{c}&
\mathbf{A}\\
\hline
& \mathbf{b}
\end{array}.
\]
For example, we list three commonly used Runge-Kutta methods in the Butcher's Tableaus in Table \ref{T:RK Butcher}. $c_i$'s are chosen from the Gaussian-Legendre collocation points. Thus, they are known as the $s$-stage Gaussian-Legendre Runge-Kutta method, namely IRK2, IRK4 and IRK6 from their order of accuracy when $s=1,2,3$, respectively.
There are many other types of collocation Runge-Kutta methods, we refer the interested reader to \cite{Cooper1987}, \cite{Sanz1988} and \cite{SA1991}.
\begin{table}[h]
\renewcommand\arraystretch{1.0}
\centering
\begin{subtable}[t]{0.3\textwidth}
\[
\begin{array}{c|c}

\frac{1}{2}   &   \frac{1}{2}   \\ \hline
    &   1   \\ 
\end{array}
\]
\caption{IRK2}
\end{subtable}
     \hfil
\begin{subtable}[t]{0.3\textwidth}
\[
\begin{array}{c|cc}
   
\frac{1}{2}-\frac{1}{6}\sqrt{3}   &   \frac{1}{4}  & \frac{1}{4}-\frac{1}{6}\sqrt{3}  \\ 
\frac{1}{2}+\frac{1}{6}\sqrt{3}   &   \frac{1}{4}+\frac{1}{6}\sqrt{3}  & \frac{1}{4}  \\ \hline
    &   \frac{1}{2} & \frac{1}{2}   \\ 
\end{array}
\]
\caption{IRK4}
\end{subtable}
    \hfil

\begin{subtable}[t]{0.3\textwidth}
\[
\begin{array}{c|ccc}
   
\frac{1}{2}-\frac{\sqrt{15}}{10}   &   \frac{5}{36}  & \frac{2}{9}-\frac{\sqrt{15}}{15}&\frac{5}{36}-\frac{\sqrt{15}}{30} \\ 

\frac{1}{2}   &   \frac{5}{36}+\frac{\sqrt{15}}{24}  & \frac{2}{9} &\frac{5}{36}-\frac{\sqrt{15}}{24} \\ 

\frac{1}{2}+\frac{\sqrt{15}}{10}   &   \frac{5}{36} + \frac{\sqrt{15}}{30}  & \frac{2}{9}+\frac{\sqrt{15}}{15}&\frac{5}{36} \\ 

 \hline
    &   \frac{5}{18} & \frac{4}{9} & \frac{5}{18}   \\ 
\end{array}
\]
\caption{IRK6}
\end{subtable}

\caption{\label{T:RK Butcher}Butcher's Tableaus for the $s$-stage Gaussian-Legendre collocation Runge-Kutta methods with $s=1,2,3$.}
\end{table}

Now, we adapt the collocation Runge-Kutta methods into the equation \eqref{E:gKdV sav}. 
}
Denote the semi-discretization in time $u^m \approx u(x,t_m)$, $v^m \approx v(t_m)$. For simplicity, we rewrite the equation \eqref{E:gKdV sav} as the system
\begin{align}\label{E:gKdv f g}
\begin{cases}
u_t=f(u,v),\\
v_t=g(u,v).
\end{cases}
\end{align}
Denote the intermediate values $U_i=u(t_m+c_i\tau)$ and $V_i= v(t_m+c_i\tau)$ to be the solution satisfying \eqref{E:gKdV sav} at the time $t=t_m+ c_i \tau$.
Then, $U_i$ and $V_i$ can be obtained from the relation
\begin{align}\label{E:RK U V}
U_i={u}^m+\tau \sum_{j=1}^{s} a_{ij}{f}_j, \quad
V_i={v}^m+\tau \sum_{j=1}^{s} a_{ij}{g}_j, 
\end{align}
where ${f}_i=f(U_i,V_i)$ and ${g}_i=g(U_i,V_i)$. Then, ${u}^{m+1}$ and $v^{m+1}$ are updated by 
\begin{align}\label{E:RK u v}
{u}^{m+1}={u}^m+\tau \sum_{i=1}^{s} b_i{f}_i, \quad
v^{m+1}=v^m+\tau \sum_{i=1}^{s} b_i {g}_i.
\end{align}

Denote
$I^{m}=I[u^m], \,M^{m}=M[u^m]$ and $E^{m}=E[u^m,v^m]$
to be the invariant quantities from \eqref{E:momentum sav}-\eqref{E:energy sav} at the time $t=t_m$. We have the following theorem.

\begin{theorem}\label{T:SRK conservation}
The $s$-stage collocation Runge-Kutta method preserves the momentum \eqref{E:momentum sav}, i.e.,
$$I^{m+1}=I^{m}.$$
Furthermore, the $s$-stage symplectic Runge-Kutta method (which is also called the quadratic preserving Runge-Kutta method) satisfies 
\begin{align}\label{E:SRK ab}
b_ia_{ij}+b_ja_{ji}-b_ib_j=0, \qquad \mbox{for} \quad i,j=1,\cdots,s,
\end{align} 
conserves the mass \eqref{E:mass sav} and energy \eqref{E:energy sav} exactly in the discrete time flow, i.e.,
$$M^{m+1}=M^{m}, \quad \mbox{and} \quad E^{m+1}=E^{m}.$$
\end{theorem}

\begin{proof}
The proof is standard, similar to \cite{Cooper1987}, \cite{Sanz1988} and \cite{LZQS2019}.
Putting \eqref{E:RK u v} in \eqref{E:momentum sav} yields
\begin{align*}
I^{m+1}-I^m=(u^{m+1}-u^m,1)= \tau(\sum_{i=1}^s b_i {f}_i,1)=\tau \sum_{i=1}^s b_i ({f}_i,1)=0.
\end{align*}
The last equality comes from the conservation law $(f_i,1)=0$ from \eqref{E:gKdV sav} at each collocation time $t=t_m+c_i \tau$.

Similarly, putting \eqref{E:RK u v} in \eqref{E:mass sav}, and using \eqref{E:RK U V}, yields
\begin{align*}
M^{m+1}&=({u}^{m+1},{u}^{m+1})=({u}^{m}+\tau \sum_{i=1}^{s} b_i{f}_i,{u}^{m}+\tau \sum_{i=1}^{s} b_i{f}_i)\\
&=({u}^m,{u}^m)- 2\tau \sum_{i=1}^s b_i({f_i},\tau \sum_{j=1}^s a_{ij}{f}_j-U_i)+ \tau^2 \sum_{i,j=1}^s b_ib_j (f_i, f_j) \\
&= M^m+2  \tau \sum_{i=1}^s b_i({f}_i,U_i)- \tau^2 \sum_{i,j=1}^s(b_ia_{ij}+b_ja_{ji}-b_ib_j)({f}_i, {f}_j)=M^m,
\end{align*}
since $(f_i,U_i)=(u_t,u)=\frac{1}{2} \frac{d}{dt}M=0$ at each collocation time $t=t_m+c_i \tau$ from the conservation law in \eqref{E:mass sav}.

Similarly, to show $E^{m+1}=E^m$, we have
\begin{align*}
E^{m+1}&=-\frac{1}{2}(u_{xx}^{m+1},{u}^{m+1})-\frac{1}{p(p+1)}((v^{m+1})^2-C_0) \\
&=-\frac{1}{2}(({u}^{m}+\tau \sum_{i=1}^{s} b_i{f}_i)_{xx},{u}^{m}+\tau \sum_{i=1}^{s} b_i{f}_i)-\frac{1}{p(p+1)}((v^m+\tau \sum_{i=1}^{s} b_i {g}_i)^2-C_0)\\
&=-\frac{1}{2}(u_{xx}^m,{u}^m)-\frac{1}{p(p+1)} \left((v^m)^2-C_0 \right)
+ \tau \sum_{i=1}^s b_i (\partial_x^2 f_i, U_i)
+\tau \sum_{i=1}^s b_i \left(\frac{2}{p(p+1)}V_ig_i \right)\\
&+\tau^2 \sum_{i,j=1}^s (b_ia_{ij}+b_ja_{ji}-b_ib_j)\left(( \partial_x^2 f_i,f_j)+{\frac{1}{p(p+1)}} g_ig_j \right)\\
&=E^m
\end{align*}

by $(\partial_x^2 f_i, U_i)+ \frac{2}{p(p+1)}V_ig_i=0$ for each $i$, which follows from differentiating the conservation law \eqref{E:energy sav} with respect to $t$. This completes the proof.
\end{proof}

From the theorem, we know that any arbitrary order of collocation Runge-Kutta method will conserve the momentum \eqref{E:momentum sav}, the symplectic Runge-Kutta method will conserve the mass \eqref{E:mass sav} and modified energy \eqref{E:energy sav} in the discrete time flow. Therefore, an arbitrarily high order time integrator can be constructed.
In fact, from the proof of Theorem \ref{T:SRK conservation}, reformulating the original equation \eqref{E:gKdV} into the system \eqref{E:gKdV sav} is only for the purpose of the energy conservation.

\subsection{A $C_0$ adjustment process}
The key part for this SAV approach is to make sure that the term $\int u^{p+1} dx +C_0$ is positive during the computational time $t \in [0,T]$. When considering the mKdV ($p=3$) case, $\int u^{p+1} dx +C_0>0$ will automatically hold for any positive constant $C_0$, since $p+1=4$ is even. This is also true for considering the nonlinear Schr\"odinger (NLS) equations or the Gross-Pitaevskii (GP) equations in \cite{CWJ2021} and \cite{LZQS2019}, as the authors there only consider the cubic nonlinearity $|u|^{2}u$. However, when considering the KdV ($p=2$) case, {a prior bound for $\int u^{p+1}(t) dx$ is needed for all $t \in [0,T]$. Otherwise, $\int u^{p+1}(t_m) dx+C_0<0$ may happen at some time $t=t_m$ (though $\int u^{p+1}(t_0) dx+C_0>0$ is satisfied in the beginning), and result in the failure of the algorithm. Instead of choosing $C_0$ carefully in the beginning of the simulation, we introduce a $C_0$ adjustment process when the value $\int u^{p+1}(t_m) dx+C_0$ is approaching $0$, and thus, we only need to choose the constant $C_0$ such that $\int u^{p+1}(t_0) dx+C_0>0$ in the beginning of the computation.
}

Suppose at $t=t_m$, $\int (u^m)^{p+1} dx +C_0<Tol$, where $Tol$ is a given positive number (e.g., $Tol=5$). Then, we choose another constant $\tilde{C_0}$ such that $\int (u^m)^{p+1} dx +\tilde{C_0}> Tol$. For example, we can take $\tilde{C_0}=10-\int (u^m)^{p+1} dx$, which leads to our new $\tilde{v}^m \approx \sqrt{10}$. Then, by using $E[u^m,v^m]=E[u^m,\tilde{v}^m]$ from \eqref{E:energy sav}, we have our new $\tilde{v}^m$
\begin{align}\label{E:v new}
\tilde{v}^m=\sqrt{(v^m)^2+\tilde{C_0}-C_0}.
\end{align}
Finally, we substitute the $v^m$ and $C_0$ in \eqref{E:KdV space} with $\tilde{v}^m$ and $\tilde{C_0}$, and then, continue with the time evolution for $t=t_{m+1},t_{m+2}, \cdots$.
 
\begin{remark}
Note that $v^2=\int u^{p+1} dx +C_0$ holds only at the collocation points $t=t_m+\tau c_i$ for each $i=1,2,\cdots,s$ in $t\in [t_m,t_{m+1}]$. However, the constant $c_s$ may not necessarily equal to $0$ or $1$, e.g., see Table \ref{T:RK Butcher}, which means $v^2=\int u^{p+1} dx +C_0$ does not hold at $t_m$ in the discrete time flow.
Therefore, the new $\tilde{v}^m$ can only be evaluated by \eqref{E:v new} to keep the discrete energy \eqref{E:energy sav} invariant. { In the actual computation, due to the dispersive nature of the negative data (e.g., Example 3), the $C_0$ adjustment process has not been executed yet.}
\end{remark}

\section{Spatial discretization}
In this section, we describe the spatial discretization. We show that the Fourier pseudo-spectral method will keep the momentum and energy invariant under such spatial discretization. Together with the temporal discretization in the previous section, the proposed scheme will preserve the discrete momentum and energy exactly in the discrete time flow. Moreover, we give an upper bound for the error of the discrete mass, which is caused by the spatial discretization.

Without loss of generality, we truncate the whole space into a bounded domain $x \in [-L,L]$ for sufficiently large $L$ with periodic boundary conditions at the boundary.
We use the Fourier pseudo-spectral method to discretize the space because of the high-order accuracy and the application of the fast Fourier transform (FFT) (see. e.g., \cite{STL2011} and \cite[Chapter 3]{Tr2000}).

Now, we briefly introduce the spatial discretization strategy. Let $N$ be the number of nodes and $h=2L/N$ to be the spatial step size. Denote $x_j=hj$ and $u_j \approx u(x_j)$ to be the spatial discretization for $j=-N/2, \cdots,N/2-1$. By applying the standard discrete Fourier expansion, one obtains
\begin{align}\label{E:DFT}
u_j= \frac{1}{\sqrt{N}} \sum_{k=-N/2}^{N/2-1} \hat{u}_k e^{i2k\pi  x_j/L}, \quad \mbox{where} \quad \hat{u}_k=\frac{1}{\sqrt{N}} \sum_{j=-N/2}^{N/2-1} u_j e^{-i2k\pi x_j/L}.
\end{align}
By introducing vectors $\mathbf{u}=(u_{-N/2}, \cdots , u_{N/2-1})^T$, $\mathbf{\hat{u}}=(\hat{u}_{-N/2}, \cdots , \hat{u}_{N/2-1})^T$, and the discrete Fourier transform matrices
$$\mathbf{F}_{k,j}=\frac{1}{\sqrt{N}} e^{-i2k \pi x_j/L}, \quad \mathbf{F}^{-1}_{j,k}= \frac{1}{\sqrt{N}} e^{i2k \pi x_j/L}, \quad -N/2 \leq j,k \leq N/2-1,$$
one can write $\mathbf{u}$ and $\mathbf{\hat{u}}$ in a matrix form as $\mathbf{u}=\mathbf{F}^{-1} \mathbf{\hat{u}}$ and $\mathbf{\hat{u}}=\mathbf{Fu}$. Also note that $\mathbf{F}^{-1}=(\mathbf{\bar{F}})^T.$
Using the properties of the Fourier transform, one can easily obtain
\begin{align*}
u'_j=\frac{i 2\pi}{L} \sum_{k=-N/2}^{N/2-1} k \hat{u}_k e^{i 2k\pi  x_j/L}, \quad
u''_j=-\left(\frac{ 2\pi}{L} \right)^2 \sum_{k=-N/2}^{N/2-1} k^2 \hat{u}_k e^{i 2k\pi  x_j/L}.
\end{align*}
Thus, we have the differential matrices $\mathbf{D_1}$ and $\mathbf{D_2}$:
\begin{align}\label{E:fft dmatrix}
\mathbf{u}'=\mathbf{D_1u}=\mathbf{\bar{F}^T \Lambda_1 F u}, \quad \mbox{and} \quad \mathbf{u}''=\mathbf{D_2u}=\mathbf{\bar{F}^T \Lambda_2 F u},
\end{align}
where $\mathbf{\Lambda_1}=i \mathbf{\Lambda}$, $\mathbf{\Lambda_2}=\mathbf{\Lambda^2_1}$, and $\mathbf{\Lambda}= \frac{2\pi}{L}\mbox{diag}(\frac{-N}{2}, \cdots,\frac{N-1}{2})$. We note that $\mathbf{D_{1,2}}$ are real matrices. Moreover, the matrix $\mathbf{D_2}$ is symmetric. The matrix $\mathbf{D_1}$ is antisymmetric ($\mathbf{D_1}=-\mathbf{D_1^T}$)  and circulant, i.e., $\mathbf{D_1}_{(j,k)}=d_{j-k}$. Also note that both the row sum and the column sum of the matrix $\mathbf{D_1}$ equal to $0$, i.e., for each fixed $k$ or $j$,
\begin{align}\label{E:D1 row sum}
\sum_{j=-N/2}^{N/2-1}d_{j,k}=\sum_{k=-N/2}^{N/2-1}d_{j,k}=0.
\end{align}
This has been studied in many literatures, see e.g., \cite{GCW2014} and \cite[Chapter 3]{Tr2000}.
We remark here that in the actual computation, $\mathbf{F}$ and $\mathbf{F}^{-1}$ are not explicitly assembled, instead, the fast Fourier transform (FFT) is typically used in the computation.

For simplicity, if $\mathbf{u}$ is a vector, we denote $\mathbf{u^p}=(u_{-N/2}^p,u_{-N/2+1}^p,\cdots u_{N/2-1}^p)^T$ to be the pointwise power. Then, we define the discrete norms. Given two vectors $\mathbf{u},\mathbf{v}$, define
\begin{align}\label{D:discrete norm}
(\mathbf{u},\mathbf{v})_h= h \mathbf{\bar{v}}^T\mathbf{u}= h \sum_{j=-N/2}^{N/2-1} u_j \bar{v}_j,\quad \| \mathbf{u} \|_h=\sqrt{(\mathbf{u},\mathbf{u})_h}.
\end{align}
Note that when $\mathbf{u},\mathbf{v} \in \mathbb{R}$, $(\mathbf{u},\mathbf{v})_h=(\mathbf{v},\mathbf{u})_h$. It is easy to see that $(\mathbf{D_2u},\mathbf{u})_h= (\mathbf{u},\mathbf{D_2u})_h  \in \mathbb{R}$, since $\mathbf{D_2}$ is real and symmetric, and $(\mathbf{D_1u},\mathbf{u})_h= (\mathbf{u},\mathbf{D_1u})_h=0$, since $\mathbf{D_1}$ is antisymmetric.

Denote $v_h$ to be the spatial semi-discretization for $v(t)$, i.e., $v_h \approx v(t)$. Then, the system of equations \eqref{E:gKdV sav} is discretized into the following system
\begin{align}\label{E:KdV space}
\begin{cases}
\mathbf{u}_t=-\mathbf{D_1}\left( \mathbf{D_2u}+\frac{1}{p}\frac{\mathbf{u^p}v_h}{\sqrt{(\mathbf{u^p},\mathbf{u})_h+C_0}} \right) \equiv f(\mathbf{u},v) ,\\
(v_h)_t=\frac{p+1}{2\sqrt{(\mathbf{u^p},\mathbf{u})_h+C_0} } ( \mathbf{u^p},\mathbf{u}_t )_h \equiv g(\mathbf{u},v) .
\end{cases}
\end{align}
 Note that $\mathbf{u}$ and $f(\mathbf{u},v_h)$ are an $N\times 1$ vectors, and $v_h$ and $g(\mathbf{u},v_h)$ are scalars.

Now, the spatial discrete momentum, mass and energy are defined as follows
\begin{align}
& I_h(\mathbf{u})=(\mathbf{u},\mathbf{1})_h; \label{D:DI S}\\
& M_h(\mathbf{u})=(\mathbf{u},\mathbf{u})_h; \label{D:DM S}\\
& E_h(\mathbf{u},v_h)=-\frac{1}{2} (\mathbf{D_2u},\mathbf{u})_h-\frac{1}{p(p+1)}(v_h^2-C_0) \label{D:DE S},
\end{align}
where $\mathbf{1}=(1,1, \cdots,1)^T$ is an $N\times 1$ vector.
We next obtain the conservation of the discrete momentum and energy, together with the error on the discrete mass.
\begin{theorem}\label{T:DM DE dt} The Fourier pseudo-spectral discretization to the equation \eqref{E:KdV space} leads to the following properties:
\begin{align}\label{E:DM DE dt}
\frac{d}{dt} I_h(\mathbf{u})=0   ; \quad \frac{d}{dt} E_h(\mathbf{u},v_h)=0.
\end{align}
Moreover,
\begin{align}\label{E:DMass dt}
\frac{d}{dt} M_h(\mathbf{u})=-\frac{2h}{p}\mathbf{u^TD_1u^p}.
\end{align}
\end{theorem}
\begin{proof}
Note that $I_h \in \mathbb{R}$, thus,
\begin{align*}
\frac{d}{dt} I_h(\mathbf{u})= (\mathbf{u}_t,\mathbf{1})_h =(\mathbf{1},\mathbf{u}_t).
\end{align*}
We note that from \eqref{E:D1 row sum}, we have
$$\mathbf{1^TD_1}=\left(\sum_{k=-N/2}^{N/2-1}d_{-N/2,k}, \sum_{j=-N/2}^{N/2-1}d_{-N/2+1,k},\cdots,\sum_{j=-N/2}^{N/2-1}d_{N/2-1,k} \right)=\mathbf{0^T}.$$
Therefore, substitute $\mathbf{u}_t$ by \eqref{E:KdV space} yields
\begin{align}\label{E:DI dt decompose}
(\mathbf{u}_t,\mathbf{1})_h=h \mathbf{1^T}\mathbf{u}_t=-h \mathbf{1^T}\mathbf{D_1}\left( \mathbf{D_2u}+\frac{1}{p}\frac{\mathbf{u^p}v_h}{\sqrt{(\mathbf{u^p},\mathbf{u})_h+C_0}} \right)=0.
\end{align}

Next, we prove $\frac{d}{dt}E_h=0$, provided by $v_h(v_h)_t=\frac{p+1}{2}(\mathbf{u^p},\mathbf{u}_t )_h$ from \eqref{E:KdV space}, i.e.,
\begin{align}\label{E:energy c dt}
 \frac{d}{dt}E_h&=-(\mathbf{D_2u},\mathbf{u}_t)_h - \frac{2}{p(p+1)} v_h (v_h)_t
= -\left( (\mathbf{D_2u},\mathbf{u}_t)_h +\frac{1}{p} (\mathbf{u^p},\mathbf{u}_t)_h  \right)  \\ 
&=  \left( \mathbf{D_2u}+\frac{1}{p}\mathbf{u^p}, \mathbf{D_1}(\mathbf{D_2u}+\frac{1}{p}\mathbf{u^p}) \right)_h =0,  \nonumber 
\end{align} 
since $\mathbf{D_1}$ is antisymmetric. 

For proving \eqref{E:DMass dt}, we first note that $\mathbf{D_1D_2}=\mathbf{D_2D_1}$ from the decomposition in \eqref{E:fft dmatrix}. Then, $(\mathbf{D_1D_2u},\mathbf{u})_h=-(\mathbf{D_2u},\mathbf{D_1u})_h$ from the antisymmetry of $\mathbf{D_1}$. On the other hand, $(\mathbf{D_1D_2u},\mathbf{u})_h=(\mathbf{D_2D_1u},\mathbf{u})_h=(\mathbf{D_2u},\mathbf{D_1u})_h$. Thus, $(\mathbf{D_1D_2u},\mathbf{u})_h=0$. Also, note that $\frac{v_h}{\sqrt{(\mathbf{u^p},\mathbf{u})_h+C_0}}=1$ in the continuous time.
Then,
\begin{align}\label{E:mass dt decompose}
 \frac{d}{dt}M_h&=2(\mathbf{u_t},\mathbf{u})_h=-2(\mathbf{D_1D_2u},\mathbf{u})_h-\frac{2}{p}(\mathbf{D_1u^p},\mathbf{u})_h\\
&=-\frac{2}{p}(\mathbf{D_1u^p},\mathbf{u})_h=-\frac{2h}{p} \mathbf{u^TD_1u^p}, \nonumber
\end{align}
which completes the proof.

\end{proof}

Theorem \ref{T:DM DE dt} shows that in the continuous time flow, the spatial discretization will keep the momentum \eqref{E:momentum sav} and energy \eqref{E:energy sav} invariant in their corresponding discrete sense. However, the time derivative on the discrete mass \eqref{E:DMass dt} is not necessarily $0$.
We next show that the proposed scheme conserves the discrete momentum and energy, and also give an error estimate of the discrete mass.
From Theorem \ref{T:DM DE dt}, we have the following little corollary.
\begin{corollary}\label{C:DM DE dt}
Suppose $(\mathbf{u},v_h)$ is the solution to \eqref{E:KdV space}, then we have
\begin{align}
&({f}(\mathbf{u},v_h),\mathbf{1})_h=0, \label{E:I fu} \\
-&(\mathbf{D_2u},{f}(\mathbf{u},v_h))_h- \frac{2}{p(p+1)} v_hg(\mathbf{u},v_h)=0, \label{E:E fu} \\
&({f}(\mathbf{u},v_h),\mathbf{u})_h=-\frac{2h}{p} \mathbf{u^TD_1u^p}. \label{E:M fu}
\end{align}
\end{corollary}
\begin{proof}
Substituting $\mathbf{u}_t={f}(\mathbf{u},v_h)$ in the equation \eqref{E:DI dt decompose} and \eqref{E:mass dt decompose} yields \eqref{E:I fu} and \eqref{E:M fu}, respectively.
From \eqref{E:energy c dt}, we have 
\begin{small}
\begin{align*}
0=\frac{d}{dt}E_h= -(\mathbf{D_2u},\mathbf{u}_t)_h - \frac{2}{p(p+1)} v_h (v_h)_t=-(\mathbf{D_2u},\mathbf{f}(\mathbf{u},v_h))_h- \frac{2}{p(p+1)} v_hg(\mathbf{u},v_h),
\end{align*}
\end{small}
which completes the proof. 
\end{proof}

Define the full discretization in space and time $u_j^m \approx u(x_j,t_m)$ and $v_h^m \approx v(t_m)$. Then, denote $\mathbf{u}^m$ and $v_h^m$ to be the numerical solution to \eqref{E:KdV space}. Define the discrete momentum, mass and energy as follows:
\begin{align}
& I_h^m=(\mathbf{u}^m,\mathbf{1})_h; \label{D:DI} \\
& M_h^m=(\mathbf{u}^m,\mathbf{u}^m)_h; \label{D:DM} \\
& E_h^m=-\frac{1}{2}(\mathbf{D_2u}^m,\mathbf{u}^m)_h-\frac{1}{p(p+1)}\left( (v_h^m)^2-C_0\right). \label{D:DE}
\end{align}
The following theorem shows the conservation of the discrete momentum \eqref{D:DI} and the discrete energy \eqref{D:DE} for the proposed scheme. It also shows the upper bound of the error for the discrete mass \eqref{D:DM}.

\begin{theorem}\label{T:DI DE DM error}
Suppose the discretized system \eqref{E:KdV space} is solved by the $s$-stage symplectic Runge-Kutta method which satisfies \eqref{E:SRK ab}, then the solution $\mathbf{u}^m$ and $v_h^m$ to the equation \eqref{E:KdV space} satisfies
\begin{align}\label{E:DI DE conserve}
I_h^{m+1}=I_h^m; \quad E_h^{m+1}=E_h^m.
\end{align}
In the meanwhile, denote $U_i^m$ and $V_i^m$ to be the intermediate step for the SRK method at $t=t_m+c_i\tau$. Also note that $U_i^m \approx \mathbf{u}^m$, then
\begin{align}\label{E:DM error}
|M_h^{m}-M_h^0| \leq t_m \frac{4h}{p} \max_{ 1 \leq \nu \leq m,1 \leq i \leq s} |(U^{\nu}_i)^{\mathbf{T}} \mathbf{D_1} (U^{\nu}_i)^{\mathbf{p}}| \approx t_m \frac{4h}{p} \max_{1 \leq \nu \leq m} |(\mathbf{u^T})^{\nu} \mathbf{D_1} (\mathbf{u^p})^{\nu}|.
\end{align}
\end{theorem}
\begin{proof}
The conservation \eqref{E:DI DE conserve} are obtained from substituting the continuous variables and inner products in Theorem \ref{T:SRK conservation} with the spatial discrete sense \eqref{D:discrete norm}, and substitute $\mathbf{u}$ in the equations \eqref{E:I fu} and \eqref{E:E fu} in Corollary \ref{C:DM DE dt} by $U_i^m$ and $V_i^m$, since they are intermediate steps which satisfy the equation \eqref{E:KdV space}.

For \eqref{E:DM error}, by substituting the continuous variables and inner products into the spatial discrete sense and use \eqref{E:M fu} yields
\begin{align*}
M_h^m-M_h^{m-1}& = 2\tau \sum_{i=1}^s b_i (f_i,U_i^{m-1})_h=-2\tau \sum_{i=1}^s b_i  \frac{2h}{p}(U^{m-1}_i)^{\mathbf{T}} \mathbf{D_1}(U^{m-1}_i)^{\mathbf{p}}  \\
&\leq \frac{4\tau h}{p} \max_{1\leq i \leq s}|(U^{m-1}_i)^{\mathbf{T}} \mathbf{D_1}(U^{m-1}_i)^{\mathbf{p}}|,
\end{align*}
since $\sum_{i=1}^sb_i=1$. Therefore,
\begin{small}
\begin{align*}
|M_h^{m}-M_h^0|&=|\sum_{\nu =1}^m (M_h^{\nu}-M_h^{\nu-1}) | \leq  \sum_{\nu =1}^m |M_h^{\nu}-M_h^{\nu-1}| \leq \frac{4\tau h}{p} \sum_{\nu =1}^m \max_{1\leq i \leq s}|(U^{\nu}_i)^{\mathbf{T}} \mathbf{D_1}(U^{\nu}_i)^{\mathbf{p}}|  \\
& \leq \frac{4\tau h}{p} \sum_{\nu =1}^m \max_{ 1 \leq \nu\leq m, 1\leq i \leq s}|(U^{\nu}_i)^{\mathbf{T}} \mathbf{D_1}(U^{\nu}_i)^{\mathbf{p}}|=\frac{4 t_m h}{p} \max_{ 1 \leq \nu\leq m, 1\leq i \leq s}|(U^{\nu}_i)^{\mathbf{T}} \mathbf{D_1}(U^{\nu}_i)^{\mathbf{p}}|, \nonumber
\end{align*}
\end{small}
the last equality comes from $t_m=m\tau$ and completes the prove.
\end{proof}
Theorem \ref{T:DI DE DM error} shows that the discrete momentum and energy will be preserved exactly in the proposed scheme. In the meanwhile, the inequality \eqref{E:DM error} shows the upper bound of the error for the discrete mass, which grows linearly with respect to time. However, note that the term $h\mathbf{u^T} \mathbf{D_1} \mathbf{u^p}$ is the Fourier spectral approximation of the integral $\frac{1}{p+1}\int (u^{p+1})_x dx$, i.e.,
$$h\mathbf{u^T} \mathbf{D_1} \mathbf{u^p} \approx ((u^p)_x,u)=\frac{1}{p+1}((u^{p+1})_x,1)=0.$$  
Hence, instead of decreasing in the 1st order with respect to $h$, the term $|h\mathbf{u^T} \mathbf{D_1} \mathbf{u^p}|$ is of spectral accuracy and very small (generally on the order of $10^{-14}$), which is usually several order lower than the tolerance of the fixed point solver for the resulting nonlinear system. { Moreover, similar to previous results observed in \cite{LY2016} for the discontinuous Galerkin scheme, despite of the failure of the mass conservation, the error of discrete mass remains at the same level of the conserved quantities (discrete momentum and energy) in the simulation. These conserved discrete momentum and energy are likely to imply the cancellations of the mass error.}
As a consequence, the error of the discrete mass becomes unnoticeable during our simulation even up to the time $t_m=200$, see details in the next section.

\section{Numerical results}
\subsection{A fast solver}
In this section, we show numerical examples for the proposed scheme. Before illustrating the examples, we describe a fast solver for solving the resulting nonlinear system from the IRK methods. This fast solver is similar to the one in \cite{CWJ2021}, which is based on the fixed point iteration. It is on the same order of the computational cost in solving the original equation, despite of the new auxiliary variable being introduced. To be specific, we consider the IRK4 ($s=2$) case. The other cases can be easily generalized.
We introduce the auxiliary variables and rewrite the system as
\begin{align}\label{E:sav solve 1}
\begin{cases}
U_i=\mathbf{u}^m+\tau \sum_{j=1}^{2} a_{ij}\mathbf{f}_j,  \\
\Phi_i=\frac{(U_i)^p}{\sqrt{((U_i)^p,U_i)_h +C_0}},\quad g_i=\frac{p+1}{2}(\Phi_i,\mathbf{f}_i)_h, \\
V_i={v_h}^m+\tau \sum_{j=1}^{2} a_{ij}{g}_j, \quad i=1,2,
\end{cases}
\end{align}
and 
\begin{align}\label{E:sav solve 2}
\begin{cases}
\mathbf{f}_1=-\mathbf{D_1}\left(\mathbf{D_2} (\mathbf{u}^m+\tau \sum_{j=1}^{2} a_{1j}\mathbf{f}_j) +\frac{1}{p}\Phi_1V_1) \right), \\
\mathbf{f}_2=-\mathbf{D_1}\left(\mathbf{D_2} (\mathbf{u}^m+\tau \sum_{j=1}^{2} a_{2j}\mathbf{f}_j) +\frac{1}{p}\Phi_2V_2) \right).
\end{cases}
\end{align}
Then, $\mathbf{u}^{m+1}$ and $v_h^{m+1}$ can be updated by \eqref{E:RK u v}. 

The system \eqref{E:sav solve 1} and \eqref{E:sav solve 2} can be solved by the fixed point iteration. At the $l$th iteration, we have
\begin{align}\label{E:sav solve 3}
\begin{cases}
\left(\mathbf{I}+\tau a_{11} \mathbf{D_3} \right) \mathbf{f}_1^{l+1}+\tau a_{12} \mathbf{D_3f}_2^{l+1}  =-(\mathbf{D_3u}^m-\frac{1}{p}\mathbf{D_1}(\Phi_1^lV_1^l)),\\
\tau a_{21} \mathbf{D_3f}_1^{l+1} + \left(\mathbf{I}+\tau a_{22} \mathbf{D_3} \right)\mathbf{f}_2^{l+1}=-(\mathbf{D_3u}^m-\frac{1}{p}\mathbf{D_1}(\Phi_2^lV_2^l)),
\end{cases}
\end{align}
where $\mathbf{I}$ is the identity matrix and $\mathbf{D_3}=\mathbf{D_1D_2}$. 
{

Note that from the system \eqref{E:sav solve 3}, each block in the matrix 
$$\mathbf{M}=\begin{pmatrix}
\mathbf{I}+\tau a_{11} \mathbf{D_3} & \tau a_{12} \mathbf{D_3} \\
\tau a_{21} \mathbf{D_3} & \mathbf{I}+\tau a_{22} \mathbf{D_3} 
\end{pmatrix}$$ 
is commute to each other (i.e., $\mathbf{M}_{i,j}\mathbf{M}_{k,l}=\mathbf{M}_{k,l}\mathbf{M}_{i,j}$). Therefore, $\mathbf{M}$ can be easily inverted: 
$$\mathbf{M}^{-1}=\begin{pmatrix}
\mathbf{J}^{-1}(\mathbf{I}+\tau a_{22} \mathbf{D_3}) & -\mathbf{J}^{-1}(\tau a_{12} \mathbf{D_3}) \\
-\mathbf{J}^{-1}(\tau a_{21} \mathbf{D_3}) & \mathbf{J}^{-1}(\mathbf{I}+\tau a_{11} \mathbf{D_3}) 
\end{pmatrix},$$
where $\mathbf{J}=(\mathbf{I}+\tau a_{11} \mathbf{D_3} )(\mathbf{I}+\tau a_{22} \mathbf{D_3})-\tau^2a_{12}a_{21}\mathbf{D_3D_3}$.

Next, note that the matrix $\mathbf{M}^{-1}$ can be further decomposed by FFT. Therefore, the system \eqref{E:sav solve 3} will be solved efficiently with the leading order of  computational cost $\mathcal{O}(2N \log(N))$. After obtain $\mathbf{f}_i^{l+1}$, we can use \eqref{E:sav solve 1} to obtain the values of other variables $U_i^{l+1}, \Phi_i^{l+1}, g_i^{l+1},$ and $ V_i^{l+1}$. The total computational cost remains on the same order. }


We set $\mathbf{f}_{1}^{0}=\mathbf{f}_{2}^{0}=f(\mathbf{u}^m,v^m_h)$ to start the iteration, and the iteration terminates when
$$\max_{i} \|\mathbf{f}_i^{l+1}-\mathbf{f}_i^{l} \|_{\infty}<\epsilon, $$ where $i=1,2$ and $\epsilon$ is set to be {$10^{-11}$ for Example 1, or $10^{-12}$ for Example 2 and 3} in our simulations.

In general, for the $s$-stage Runge-Kutta method, { we can always find the explicit $\mathbf{M}^{-1}$ for the system similar to \eqref{E:sav solve 3}}. The computational cost is consequently on the order of $\mathcal{O}(s\times N\log(N))$. The $s$-stage Gaussian-Legendre collocation methods will obtain the maximum order of $2s$ on temporal accuracy. 

{
\subsection{Brief description of existing time integrators}
Since we are going to compare our proposed methods to other time integrators, we briefly describe the other time integrators we used in this paper. For simplicity, we only discuss the temporal discretization in this subsection. The space-time full discretization can be easily adapted by applying the Fourier pseudo-spectral spatial discretization as described in the previous section.
\subsubsection{Modified Crank-Nicholson method}
We first review the widely used 2nd order conservative modified Crank-Nicolson (MCN) method as follows:
\begin{small}
\begin{align}\label{E:MCN time}
\frac{u^{m+1}-u^m}{\tau} +\left( \frac{1}{2} (u^{m+1}_{xx}+u^m_{xx}) + \frac{1}{p(p+1)} \frac{(u^{m+1})^{p+1}-(u^{m})^{p+1}}{(u^{m+1})^{2}-(u^{m})^{2}} (u^{m+1}+u^{m})  \right)_x=0.
\end{align} 
\end{small} 
The nonlinear system \eqref{E:MCN time} can be solved by the fixed point iteration in the same spirit from \eqref{E:sav solve 3}. We omit the details for the purpose of conciseness.

\begin{proposition}\label{P:MCN}
The modified Crank-Nicolson scheme \eqref{E:MCN time} preserves the following two invariants: 
$$I_h^{m+1}=I^m  \quad E_h^{m+1}=E^m.$$
Consequently, if the the scheme is discretized by the Fourier pseudo-spectral method (section 4), then the discrete momentum and energy will be preserved ($I_h^{m+1}=I_h^m$ and $E_h^{m+1}=E_h^m$). 
\end{proposition}
\begin{proof}
By taking the inner product with $1$ and the term $$\frac{1}{2} (u^{m+1}_{xx}+u^m_{xx}) + \frac{1}{p(p+1)} \frac{(u^{m+1})^{p+1}-(u^{m})^{p+1}}{(u^{m+1})^{2}-(u^{m})^{2}} (u^{m+1}+u^{m}),$$ we obtain that three invariants $I^m$ and $E^m$ are preserved in the temporal semi-discretized sense (consider the spatial variable is continuous), respectively. Furthermore, by using the argument in the previous section, we can obtain $I_h^{m+1}=I_h^m$ and $E_h^{m+1}=E_h^m$. 
\end{proof}

\subsubsection{Semi-implicit Scalar auxiliary variable method}
We next consider the original semi-implicit SAV approach from \cite{SXY2018}. The main advantage is that solving the resulting nonlinear algebraic system can be avoided. Other than the conventional Crank-Nicholson-Adam-Bashforth (CNAB) approach, we use the Leap-Frog scheme here, since it is time reversible. To be specific, we discretize the equation \eqref{E:gKdV sav} as follows
\begin{align}\label{E:SAV LP}
\begin{cases}
\frac{u^{m+1}-u^{m-1}}{2\tau}+\tilde{u}^m_{xxx}+\left(\frac{1}{p}q^m \tilde{v}^m \right)_x=0,\\
v^{m+1}-v^{m-1}=\frac{p+1}{2}(q^m,u^{m+1}-u^{m-1}),
\end{cases}
\end{align}
where $\tilde{u}^m=\frac{1}{2}(u^{m+1}+u^{m-1})$, $q^m=q(u^m)=\frac{(u^m)^p}{\sqrt{((u^m)^p,u^m)+C_0}}$, and $\tilde{v}^m=\frac{1}{2}(v^{m+1}+v^{m-1})$.

In order to solve this equation system, we rewrite the first equation in \eqref{E:SAV LP} as
\begin{align}\label{E:SAV LP2}
\frac{\tilde{u}^{m}-u^{m-1}}{\tau}+\tilde{u}^m_{xxx}+\left(\frac{1}{p}q^m \tilde{v}^m \right)_x=0,
\end{align}
Assuming $\tilde{u}^m=u_1^m+\tilde{v}^m u_2^m$, putting it in \eqref{E:SAV LP2} yields the linear decoupled system
\begin{align}\label{E:SAV LP P1}
\begin{cases}
(1+ \tau \partial_{x}^3)u_1= u^{m-1},\\
(1+\tau \partial_{x}^3)u_2=\frac{1}{p}(q^m)_x.
\end{cases}
\end{align}

Hence, $u_1$ and $u_2$ can be obtained by inverting the operator $(1+\tau \partial_{x}^3)$, which will be easy if the Fourier spectral spatial discreitzation is considered. 

Then, substituting $v^{m+1}=2\tilde{v}^n-v^{m-1}$ and putting it into the second equation of \eqref{E:SAV LP}, using the relation $u^{m+1}=2\tilde{u}^m-u^{m-1}$ and $\tilde{u}^m=u_1^m+\tilde{v}^mu_2^m$, we obtain the formula for $\tilde{v}^m$:
\begin{align}\label{E:SAV LP v}
\tilde{v}^m=\left( {1-\frac{p+1}{2}(q^m,u_2^m)}\right)^{-1} \left( \frac{p+1}{2}(q^m,u_1^m-u^{m-1})+v^{m-1}  \right).
\end{align}
The scheme \eqref{E:SAV LP} is a semi-implicit multistep method. 
The first step $u^1$ can be obtained from the MCN method \eqref{E:MCN time} above. Moreover, we can have the modified energy conserved and the error of the mass is approximated by the third order in time ($\mathcal{O}(\tau^3)$) in the following proposition.
\begin{proposition}\label{P:SAV LP}
The scheme \eqref{E:SAV LP} preserves the momentum $I^{m+1}=I^m$ and the modified energy $E^{m+1}=E^{m}$; and the local error of the mass is at the third order in time, i.e.,
\begin{align}\label{E:SAV LP mass}
M^{m+1}=M^{m}+\mathcal{O}(\tau^3).
\end{align}
\end{proposition}
\begin{proof}
Taking the inner product of the first equation \eqref{E:SAV LP} with $1$ will yield the conservation of the momentum $I^{m+1}=I^{m-1}$, which implies $I^{m+1}=I^m$. 

For the conservation of the modified energy, we take the inner product of the first equation \eqref{E:SAV LP} with the term $\tilde{u}^m_{xx}+\frac{1}{p}q^m\tilde{v}^m$. Then, substituting $\tilde{u}^m=\frac{1}{2}(u^{m+1}+u^{m-1})$ yields
$$\frac{1}{2}(u^{m+1}-u^{m-1},\tilde{u}^m_{xx})+\frac{1}{p}(u^{m+1}-u^{m-1},q^m) \cdot \tilde{v}^m=0.$$ 
Note that $(u^{m+1}-u^{m-1},q^m)=\frac{2}{p+1}(v^{m+1}-v^{m-1})$ from the second equation of \eqref{E:SAV LP}, and thus, we obtain the energy conservation relation $E^{m+1}=E^{m-1}$.

For the local error of mass, we first notice that $\tilde{u}^m=u^m+\mathcal{O}(\tau^2)$, and consequently $\tilde{q}^m:=q(\tilde{u}^m)=q^m+\mathcal{O}(\tau^2)$.
Taking the inner product of the first equation \eqref{E:SAV LP} with $\tilde{u}^m$ and substitute $q^m=\tilde{q}^m+\mathcal{O}(\tau^2)$, this yields
$$\frac{1}{2\tau}(u^{m+1}-u^{m-1}, \tilde{u}^m)+((\tilde{q}^m)_x,\tilde{u}^m)\cdot \tilde{v}^m+\mathcal{O}(\tau^2)=0.$$
Note that the term $((\tilde{q}^m)_x,\tilde{u}^m)=0$. Thus, we obtain \eqref{E:SAV LP mass} and complete the proof.
\end{proof}

One can see that unlike the SAV-IRK methods in Theorem \ref{T:SRK conservation}, the semi-implicit SAV-Leap Frog (SAV-LF) method will not preserve the mass. By following the same argument, similar results can be obtained by  for the conventional SAV-CNAB method from \cite{SXY2018}. Later in our numerical examples, we will show that the failure of mass conservation will lead to the failure of simulating the breathers.  


\subsubsection{Strang splitting method}
Another widely used method for the gKdV equations is the operator splitting method, e.g., see \cite{HKR1999} and \cite{HKRT2011}. We briefly describe the 2nd order Strang splitting (SS) method here. We consider the ODE (or PDE) of the form
\begin{align}\label{E:ODE ETD}
u_t=\mathbf{L}u+\mathbf{N}(u),
\end{align}
where $\mathbf{L}$ is the linear part ($\partial_x^3$ for gKdV), and $\mathbf{N}(u)$ is the nonlinear term ($\frac{1}{p}(u^p)_x$ for gKdV).
The spirit of operator splitting method is to split the gKdV equation \eqref{E:gKdV} into two equations. One is the linear dispersive equation 
\begin{align}\label{E:gKdV L}
u_t+u_{xxx}=0.
\end{align}
We denote the time evolution from $t=t_m$ to $t_{m+1}$ for this linear dispersive equation as $u^{m+1}_{\tau}=A_{\tau} u^m$. 

The other is the conservation law equation
\begin{align}\label{E:gKdV N}
u_t+(\frac{1}{p} u^p)_{x}=0.
\end{align}
We denote the time evolution as $u^{m+1}_{\tau}=S_{\tau} u^m$.

We use the strategy in \cite{HKR1999}.
The time evolution $u^{m+1}_{\tau}=A_{\tau} u^m$ is approximated by the standard Crank-Nicolson scheme:
\begin{align*}
(u^{m+1}-u^{m})+\frac{\tau}{2}\left( u^{m+1}_{xxx}+ u^{m}_{xxx}\right)=0,
\end{align*}
and the time evolution for $u^{m+1}_{\tau}=S_{\tau} u^m$ is approximated by the IRK2 (mid-point) method:
\begin{align*}
(u^{m+1}-u^{m})+\frac{\tau}{p} \left( \left( \frac{u^{m+1}+u^m}{2}\right)^p \right)_x=0.
\end{align*}
Again, in the same spirit of \eqref{E:MCN time}, the fixed point iteration can be used to solve the resulting nonlinear system after the spatial discretization.
Then, the solution can be updated by $u^{m+1}=S_{\frac{\tau}{2}}A_{\tau}S_{\frac{\tau}{2}}u^{m}.$

The main advantage of the operator splitting method is that we can treat the linear and nonlinear part via different kinds of approaches, and thus, leads to great flexibility. We refer the interest readers to \cite{HKR1999} and \cite{HKRT2011} for the detailed study of operator splitting methods for the gKdV equations.



\subsubsection{Exponential Time Differencing method}
The other widely used efficient numerical method for the stiff ODE's are the Exponential Time Differencing methods. 
Recall the Duhammel's form of the equation \eqref{E:ODE ETD}: 
\begin{align}\label{E:ODE Du}
u^{m+1}=e^{\mathbf{L}\tau}u^{m}+  \int_0^{\tau} e^{\mathbf{L} (\tau -s)} \mathbf{N}(u(t_m+s)) ds.
\end{align}
The linear stiff part $\mathbf{L}u$ could be solved exactly as $e^{\mathbf{L}\tau}u^{m}$ from proper spatial discretizations (e.g., Fourier pseudo-spectral method), and the numerical quadrature can be constructed to approximate the nonlinear part in \eqref{E:ODE Du}.
One possible fourth order scheme is proposed by Cox and Matthews (known as modified ETDRK4 or mETDRK4), which gives the following formulas:
\begin{align}\label{E:ETDRK}
&a_m=e^{\mathbf{L} \tau/2}u^m+ \mathbf{L}^{-1} (e^{\mathbf{L} \tau/2} - \mathbf{I})\mathbf{N}(u^m); \nonumber \\
&b_m=e^{\mathbf{L} \tau/2}u^m+ \mathbf{L}^{-1} (e^{\mathbf{L} \tau /2} - \mathbf{I})\mathbf{N}(a_m); \nonumber \\
&c_m=e^{\mathbf{L} \tau/2}a_m+ \mathbf{L}^{-1} (e^{\mathbf{L} \tau /2} - \mathbf{I})(2\mathbf{N}(b_m)-\mathbf{N}(u^m)); \nonumber \\
&\mathbf{g}_1=\tau^{-2} \mathbf{L}^{-3} [-4-\tau \mathbf{L} + e^{\tau \mathbf{L}}(4-3\tau \mathbf{L} +(\tau \mathbf{L})^2 ) ]; \nonumber \\
&\mathbf{g}_2=2\tau^{-2} \mathbf{L}^{-3} [2+\tau \mathbf{L} + e^{\tau \mathbf{L}}(-2+\tau \mathbf{L} ) ];  \nonumber \\
&\mathbf{g}_3=\tau^{-2} \mathbf{L}^{-3} [-4-3 \tau \mathbf{L} - (\tau \mathbf{L})^2+ e^{\tau \mathbf{L}}(4-\tau \mathbf{L}) ]; \nonumber \\
&u^{m+1}=e^{\tau \mathbf{L}} u^m +\mathbf{g}_1\mathbf{N}(u^m)+\mathbf{g}_2[\mathbf{N}(a_m)+\mathbf{N}(b_m)]+\mathbf{g}_3\mathbf{N}(c_m),
\end{align}
where $\mathbf{I}$ is the identity operator.
The potential singularity of the term $\mathbf{L}^{-1} (e^{\mathbf{L} \Delta t/2} - \mathbf{I})$ at the $0$th Fourier node can be resolved by using the contour integrals from \cite{KT05}.

The mETDRK scheme \eqref{E:ETDRK} is an explicit scheme with 4th order accuracy, and thus, very efficient.
For the gKdV case in this paper, we notice that this scheme \eqref{E:ETDRK} still has the CFL condition restriction $ \tau \leq Const \cdot L/N $, since the first order derivative is shown in the nonlinear term. 
}

\subsection{Numerical examples}
{In this subsection, we list our numerical examples by using the proposed schemes. We denote the $s$-stage implicit Gaussian-Legendre Runge-Kutta methods in Table \ref{T:SRK conservation} as SAV-IRK2, SAV-IRK4 and SAV-IRK6, for $s=1,2,3$, respectively.

We first compare the proposed SAV-IRK4 method with the existing (MCN, SAV-LF, SS and mETDRK4) methods by simulating the breathers, which is non-disperse oscillatory pulse, for the modified KdV (mKdV) equations. We find that only the conservative methods (SAV-IRK4 and MCN) are capable of simulating these type of solutions for long time, and the SAV-IRK4 method is the most efficient one. 

Then, we switch to the widely considered soliton interaction solutions and scattering solutions in the next two examples. In these cases, while the invariant quantities $I_h^m$, $M_h^m$ and ${E}_h^m$ could be preserved well as shown in Theorem \ref{T:SRK conservation}, the 4th order mETDRK4 method is still the most efficient one, since it is an explicit method and no iteration is needed from solving the nonlinear system.

We track the the $L^{\infty}$ error $\mathcal{E}^m$, error of discrete momentum $\mathcal{E}_I^m$, discrete mass $\mathcal{E}^m_M$ and discrete (modified) energy $\mathcal{E}^m_E$ at $t=t_m$ as follows:   
}
\begin{align*}
&\mathcal{E}^m=\|u_{exact}^m-\mathbf{u}^m\|_{\infty}; \quad
\mathcal{E}_I^m=|I_h^m-I_h^0|; \\
&\mathcal{E}^m_M=\max_{j<m}|M_h^j-M_h^0|;  \quad
\mathcal{E}^m_{E}=\max_{j<m}|E_h^j-E_h^0|. 
\end{align*}

{
{\flushleft \it \underline{Example $1$}} We first consider the breathers for the modified KdV equation ($p=3$), which describe the non-dispersing oscillatory pulses. The two parameter family of exact solutions can be written as
\begin{align}\label{E:breathers}
B_{\alpha,\beta}(x,t)= & 2 \sqrt{6} \beta \sech(\beta(x+\gamma t)) \times \nonumber \\
& \left[ \frac{\cos(\alpha(x+\delta t))-(\beta/\alpha) \sin(\alpha(x+\delta t)) \tanh (\beta(x+\gamma t)}{1+(\beta/\alpha)^2 \sin^2(\alpha(x+\delta t)) \sech^2(\beta(x+\gamma t))} \right],
\end{align}
with $\gamma=3\alpha^2-\beta^2$, and $\delta=\alpha^2-3\beta^2$.
The phase velocity of the pulse is given by $\delta$, and the group velocity by $\gamma$ traveling to the left. The authors in \cite{AM2013} showed that breathers have the orbital stability. In \cite{CAV2013}, the numerical experiments show that the numerical instability may easily occur as time evolves and the numerical error accumulates, which eventually leads to the failure of simulation. Thus, extra care is needed in simulating this type of solutions, especially for the long time simulation.

In our numerical experiment, we take $u_0=B_{\alpha,\beta}(x,0)$, on the periodic domain $[-10\pi,10\pi]$ with $N=2^{10}$ nodes. We take $\alpha=3$ and $\beta=1$. The time step $\tau$ is taking to be as large as possible such that the breathers will still behave good at the stopping time $T=1000$, but will collapse if we take twice of it. This can be an indication of the stability for different numerical schemes. Surprisingly, the SAV-LF scheme does not work for simulating the breathers. In our numerical experiment, when we take $\tau=1e-4$, the solution from the SAV-LF scheme collapse around time $t=0.55$, and the solution collapse around $t=0.76$ if we take $\tau=1e-5$. Hence, the increasing mass error (see Proposition \ref{P:SAV LP}) from this scheme will lead to the instability to the breathers.

The top left subplot in Figure \ref{F:breathers 1} illustrates the solution profiles of the breather at different time on the periodic domain. The rest of the subplots in Figure \ref{F:breathers 1} show the conservation  of the discrete momentum, mass and energy (modified energy for SAV-IRK4). These results justify the conservation properties of the schemes proposed in Theorem \ref{T:SRK conservation} and Proportion \ref{P:MCN}. Moreover, the SS and mETDRK4 schemes also preserve the discrete momentum from our simulation. 

\begin{figure}
\includegraphics[width=0.45\textwidth]{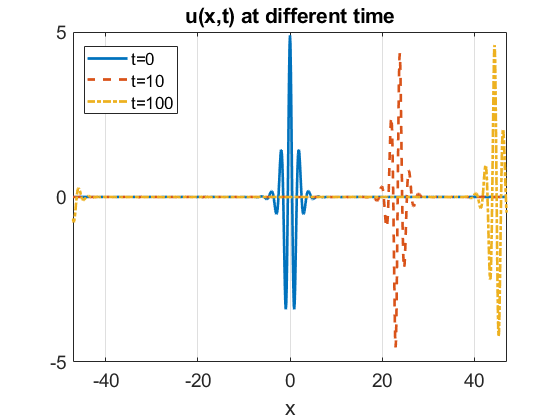}
\includegraphics[width=0.45\textwidth]{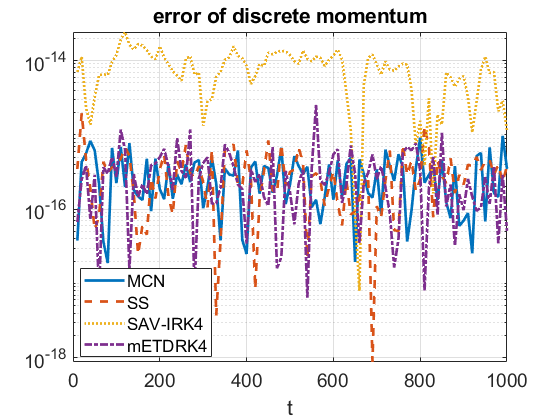}\\
\includegraphics[width=0.45\textwidth]{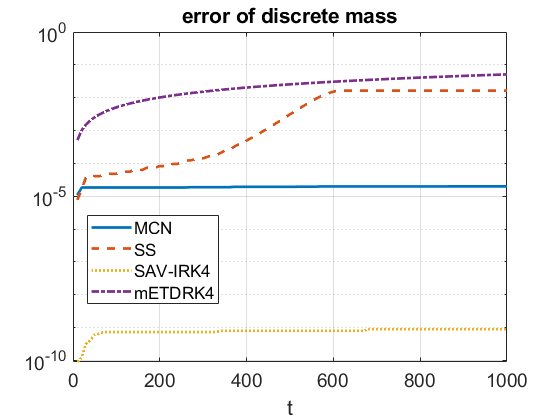}
\includegraphics[width=0.45\textwidth]{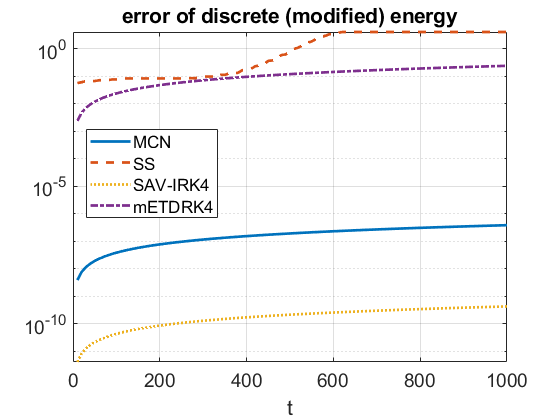}
\caption{\label{F:breathers 1} Top left: Breather profiles. The solution travels periodically on the domain $[-L,L]$. Top right: error of discrete momentum for from different time integrators. Bottom left: error of discrete mass. Bottom right: error of discrete energy (modified energy for SAV-IRK4).}
\end{figure}

From \cite{AM2013}, we have the relation $M[B]=2\beta M[Q]$, and $E[B]=2\beta \gamma |E[Q]|$, where $Q= \sqrt{6} \sech (x)$ being the soliton solution from the equation $$-Q_{xx}+Q-\frac{1}{3}Q^3=0.$$
The consistency can be checked by tracking the numerical quantities of $\beta$ and $\gamma$ with respect to time. That is, tracking 
$${\beta}_{num}^m=\frac{E^m_h}{2M[Q]} \quad \mbox{and} \quad {\gamma}_{num}^m=\frac{E^m_h}{2{\beta}_{num}^m |E[Q]|},$$
at different time $t=t_m$ and comparing with the inital given $\beta$ and $\gamma$. 
Note that the $E_h^m$ here for SAV-IRK scheme is the energy from the form \eqref{E:energy}, not the modified energy \eqref{E:energy sav}. 

\begin{figure}
\includegraphics[width=0.45\textwidth]{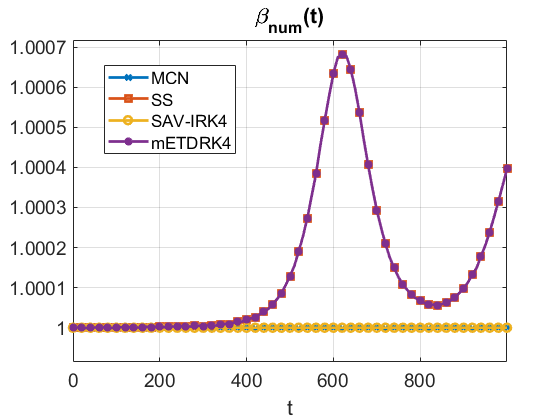}
\includegraphics[width=0.45\textwidth]{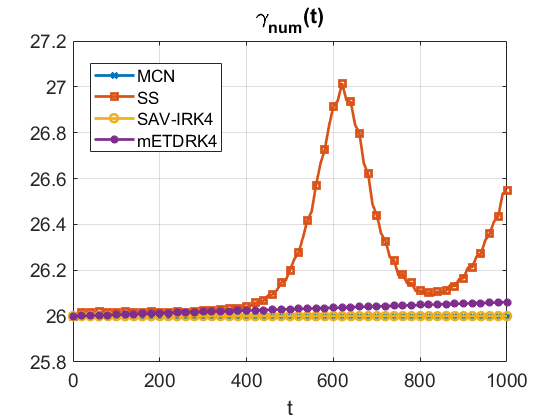}
\caption{\label{F:breathers beta} Left: $\beta_{num}(t)$ from different time integrators. Right: $\gamma_{num}(t)$ from different time integrators.}
\end{figure}

In Figure \ref{F:breathers beta}, there is no visual difference for the values of $\beta_{num}$ between SS and mETDRK4 methods, and so are for the conservative MCN and SAV-IRK4 methods, see the left subplot. When tracking $\gamma_{num}$, we observe that the SS method has the most deviation from the accurate value of $\gamma=26$. In the meanwhile, $\gamma_{num}$ from the mETDRK4 method shows a good agreement to $\gamma$, but the deviation increases linearly (see the purple circle line in the right subplot). However, the conservative methods, MCN and SAV-IRK4, keep close to the accurate value $\gamma$ all the time. 

Table \ref{T:breather} gives more details on our simulation. The simulation is processed via Matlab 2019b on the workstation with 18 cores Intel processor i9-10980XE. 
One can see that the SAV-IRK4 scheme is the most efficient one, since we can take the largest time step size $\tau=0.02$, and the quantities $\max_{m}(|\beta-\beta_{num}^m|)$ and $\max_{m}(|\gamma-\gamma_{num}^m|)$ still remain small ($\sim 10^{-11}$) at the end of simulation. While the mETDRK4 scheme (with small $\tau=0.0005$) takes less CPU time than the SAV-IRK4 scheme,  as shown in Figure \ref{F:breathers beta}, the $\max_{m}(|\gamma-\gamma_{num}^m|)=0.061$, which is relatively large and potential collapse may happen if extending the simulation time. 
Indeed, when we extend our simulation time to $T=2000$ with the same time step size in Table \ref{T:breather},   
the breather collapse from the mETDRK4 scheme (red dash), but still keeps its shape from the SAV-IRK4 method (blue solid) (see the left subplot in Figure \ref{F:breathers t2000}). In the right subplot of Figure \ref{F:breathers t2000}, the oscillation of $|\beta-\beta_{num}^m|$ and $|\gamma-\gamma_{num}^m|$ are within the range of $10^{-9}$, which suggests the possible cancellation of the numerical errors from the SAV-IRK4 scheme, other than the accumulation of the numerical errors from the mETDRK4 scheme which leads to the failure of the simulation.

\begin{center}
\begin{table}
\begin{tabular}{|c|c|c|c|c|}
\hline

  Method & $\tau$&  $\max_{m}(|\beta-\beta_{num}^m|)$  & $\max_{m}(|\gamma-\gamma_{num}|)$  & CPU time  \\
  \hline
 MCN & $2e-3$& $8e-7$ & $2e-5$ & $739.09$s  \\
 \hline
 SAV-LF& NA & NA & NA & NA   \\
 \hline
 SS& $2e-3$&$7e-4$ & $1.016$ & $882.78$s  \\
 \hline
 SAV-IRK4 & $2e-2$& $3e-11$ & $2e-11$& $389.96$s \\
 \hline
  mETDRK4 & $5e-4$& $7e-4$ & $0.061$& $287.68$s \\
 \hline
\end{tabular}
\caption{\label{T:breather} Error and computational time to $T=1000$ for simulation the breathers from different schemes.}
\end{table}
\end{center}


\begin{figure}
\includegraphics[width=0.45\textwidth]{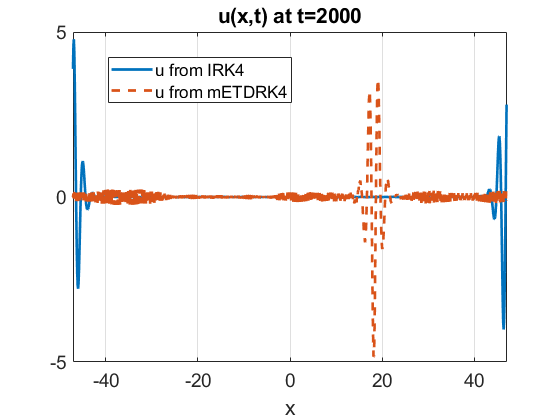}
\includegraphics[width=0.45\textwidth]{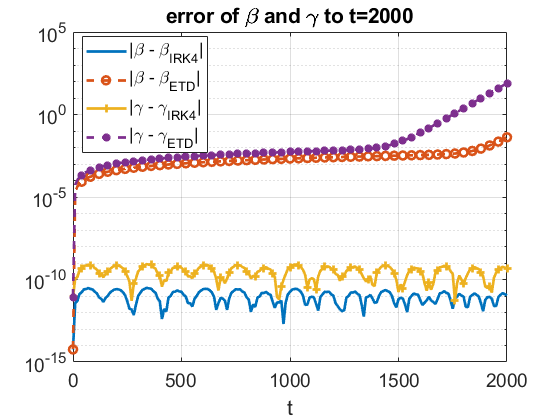}
\caption{\label{F:breathers t2000} Left: Solution profiles at $t=2000$ from SAV-IRK4 and mETDRK4. Right: $|\beta-\beta_{num}|$ and $|\gamma-\gamma_{num}|$ from different time integrators.}
\end{figure}

}

{
Our remaining examples consider the comparison between the IRK methods directly from \eqref{E:gKdV} and the SAV-IRK methods, showing that the conservative SAV-IRK approach will preserve the discrete momentum and energy as desired, regardless for the soliton type of solutions or the fast oscillating disperse soltuions. The results from the mETDRK4 method is also included as a comparison. 

\medskip

{\flushleft \it \underline{Example $2$}}. We consider the interaction of two soliton solutions to the KdV ($p=2$) equation. }
The exact solution satisfies
\begin{align}\label{E:u ex1}
u(x,t)=12\frac{\gamma_1^2e^{\theta_1}+\gamma_2^2e^{\theta_2}+2(\gamma_2-\gamma_1)^2 e^{\theta_1+\theta_2}+a^2(\gamma_2^2e^{\theta_1}+\gamma_1^2e^{\theta_2} ) e^{\theta_1+\theta_2}}{(1+e^{\theta_1}+e^{\theta_2}+a^2e^{\theta_1+\theta_2})^2},
\end{align}
where we take the parameters
\begin{align*}
&\gamma_1=0.4, \,\, \gamma_2=0.6, \,\, a^2=\left( \frac{\gamma_1-\gamma_2}{\gamma_1+\gamma_2} \right)^2=\frac{1}{25}, \\
& \theta_1=\gamma_1x-\gamma_1^3t+x_1, \,\, \theta_2=\gamma_2 x-\gamma_2^3t+x_2, \,\, x_1=10, \,\, x_2=25.
\end{align*}

This initial condition represents two solitons, a larger one is on the left of the smaller one (see the left subplot in Figure \ref{F:EX1 profile}). Both of the solitons travel to the right. The larger soliton travels with faster speed, and thus, will catch up with the smaller one. Then, they are supposed to merge and split again. 
We take $L=30\pi$ with $N=2048$ in space. The time step is taken to be $\tau=0.1$  and the stopping time $T=200$. Figure \ref{F:EX1 profile} shows the solution profile obtained by the SAV-IRK4 method. The left subplot is the initial condition $u_0$, the right subplot is the time evolution. One can see that the solitons travel to the right with different speeds, intersect, and then split again with their initial speeds. This is similar to the result in \cite{YHL2013}. 

\begin{figure}
\includegraphics[width=0.45\textwidth]{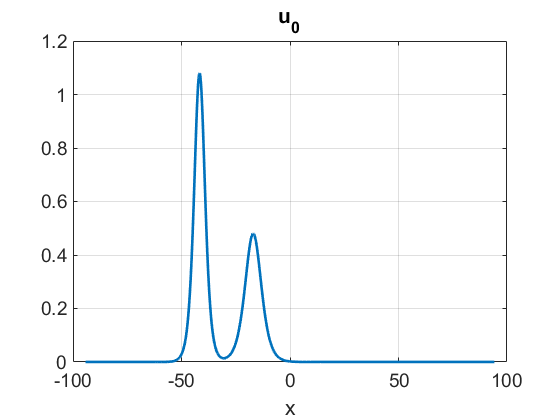}
\includegraphics[width=0.45\textwidth]{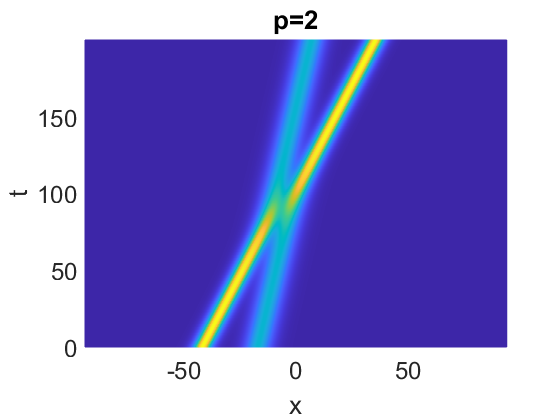}
\caption{\label{F:EX1 profile} The solution profile for Example 2 from SAV-IRK4. Left: $u_0$. Right: $u(x,t)$ from the top view.}
\end{figure}


Figure \ref{F:EX1 data} tracks the results obtained by different time integrators. The top left subplot in Figure \ref{F:EX1 data} shows $\|\mathbf{u}^m-u_{exact}\|_{\infty}$ with respect to time, where $u_{exact}=u(\mathbf{x},t_m)$ is the exact solution. One can see that the error by using the mETDRK4 method is between the error by IRK2 type methods (IRK2 and SAV-IRK4) and IRK4 type methods (IRK4 and SAV-IRK4). This is reasonable, since IRK2 type methods are of the second order accuracy in time, whereas the rest three are of the 4th order accuracy in time. { The IRK4 type methods are more accurate in simulating the soliton interactions at the same time step level. However, IRK4 type methods are implicit methods which require iterations at each time step. For example, given the same time step size, the mETDRK4 method is usually $10$ to $40$ times faster than the SAV-IRK4 scheme. Hence, for this soliton type of solutions, choosing smaller time step for the mETDRK4 method will generate more accurate results than the IRK4 methods in shorter actual computational time.}
The two IRK2 methods (IRK2 and SAV-IRK2), and the two IRK4 methods (IRK4 and SAV-IRK4) are indistinguishable from each other in the subplot. {In fact, their $L^{\infty}$ difference are on the order of $10^{-4}$ for IRK2 methods, and $10^{-9}$ for IRK4 methods.}

{
The rest of the subplots, from the top right to the bottom right  in Figure \ref{F:EX1 data}, track the errors for the discrete momentum ($I_h^m$), discrete mass ($M_h^m$) and discrete energy ($E_h^m$ from \eqref{E:energy}) from IRK and mETDRK4 methods or the modified discrete energy ($E_h^m$ from \eqref{E:energy sav}) from SAV-IRK methods. The conservation of the $I_h^m$ and modified energy ${E}_h^m$ show the good agreement to the Theorem \ref{T:SRK conservation}. We also observe that in this case, error of discrete mass in Theorem \ref{T:DI DE DM error} is very small, even smaller than our fixed point solver tolerance ($10^{-12}$), where the theoretical conservation could not be provided in Theorem \ref{T:SRK conservation}.
}

\begin{figure}
\includegraphics[width=0.45\textwidth]{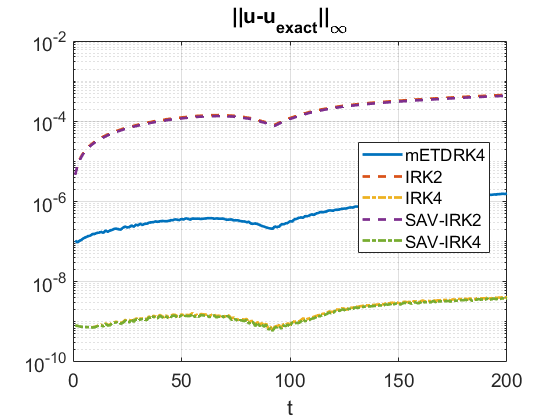}
\includegraphics[width=0.45\textwidth]{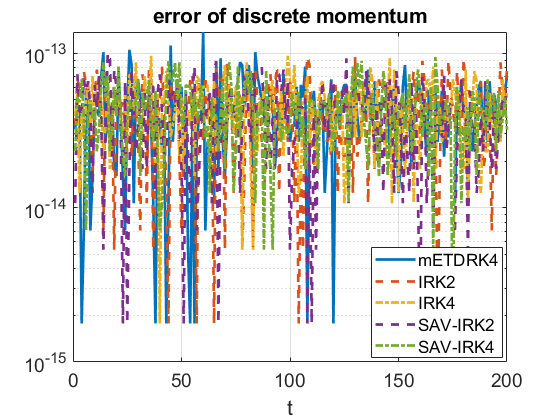}\\
\includegraphics[width=0.45\textwidth]{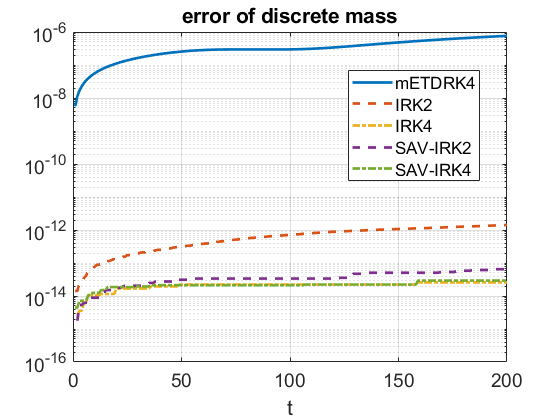}
\includegraphics[width=0.45\textwidth]{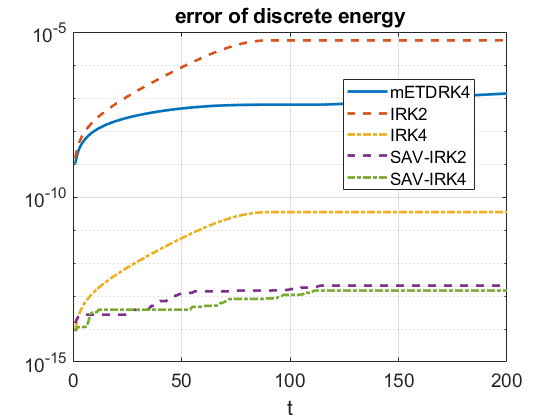}
\caption{\label{F:EX1 data} The errors in Example 2 by different time integrators: mETDRK4 (solid blue); IRK2(dash red); SAV-IRK2 (dash purple); IRK4 (dash dot orange); SAV-IRK4 (dash dot green). Top left: $\|u-u_{exact}\|_{\infty}$. Top right: discrete momentum error. Bottom left: discrete mass error. Bottom right: discrete energy error. }
\end{figure}


We also list the $L^{\infty}$ error $\mathcal{E}^m$ for the SAV-IRK2 and SAV-IRK4 methods from different time steps $\tau$ in Table \ref{T:EX1 error}. { The rate stands for the ratio of the error $\mathcal{E}^m$ in the previous row divide the error $\mathcal{E}^m$ in the current row, and etc.} 
One can see that the error for SAV-IRK2 method decreases at the second order speed, and the error for SAV-IRK4 method decreases at the 4th order speed, except for $\tau=\frac{1}{40}$, where the temporal error is too close to the error of solving the nonlinear system, and thus, affects the final results. Therefore, the SAV-IRK2 and SAV-IRK4 are still of the 2nd order and the 4th order temporal accuracy, respectively. { Moreover, by taking $\tau=\frac{1}{160}$ for the mETDRK4 scheme, we obtain the more accurate result ($\mathcal{E}^m \approx 5e-12$) in shorter computational time ($\sim 9.27$s).}


\begin{center}
\begin{table}
\begin{tabular}{|c|c|c|c|c|c|c|}
\hline
 & \multicolumn{3}{|c|}{SAV-IRK2} &  \multicolumn{3}{|c|}{SAV-IRK4}\\
\hline
  $\tau$&  error  & rate  & CPU time  & error  & rate & CPU time\\
  \hline
 $\frac{1}{5}$& $1.7e-3$ & NA & $6.579$s & $6.40e-8$ & NA & $4.949$s \\
 \hline
 $\frac{1}{10}$& $4.3e-4$& $3.998$ & $13.966$s & $3.89e-9$ &  $16.47$ & $12.016$s  \\
 \hline
 $\frac{1}{20}$&$1.1e-4$ & $3.998$ & $22.787$s & $2.47e-10$ & $15.75$&  $21.944$s \\
 \hline
 $\frac{1}{40}$& $2.72e-5$ & $4.001$& $34.330$s & $1.78e-11$ & $13.82$ & $38.257$s \\
 \hline
\end{tabular}
\caption{\label{T:EX1 error} The convergence rates of SAV-IRK2 and SAV-IRK4 in Example 2. We note that by using the mETDRK4, we can obtain the error $\approx 5e-12$ (more accurate) in $\sim 9.27$ seconds (faster).}
\end{table}
\end{center}


{\flushleft \it \underline{Example $3$}}. We next consider the scattering solution for the KdV equation with the initial condition $u_0=-\sech^2(x)$. 
This example is a soliton type data with negative sign. The KdV evolution will lead $u_0$ to the dispersion, and possible negative value for $\int u^3 dx$. Here, the $C_0$ adjustment process in \eqref{E:v new} will make $v(t)$ stay positive, and thus, will keep the algorithm applicable for all times. 
The exact solution is not explicitly given, since it is not exactly the soliton due to the negative sign and coefficients chosen. We compute the reference solution $u_{ref}$ by both mETDRK4 and SAV-IRK4 methods separately with an ultimately small time step ($\tau=1/25600$). 
The results from these two different approaches differ at the level of $10^{-12}$.  Therefore, we can take the reference solution as the ``exact" solution.
In the simulation, we take $\tau=0.01$, with $L=30 \pi$ and $N=2048$, the same as in the previous example.

The left subplot in Figure \ref{F:EX2 profile} shows the solution profile evolving in time. One can see that it disperses to the left. The right subplot shows the solution profile at the stopping time $T=1$. The fast oscillations on the left can be observed. Figure \ref{F:EX2 data} tracks the data $\mathcal{E}^m$, $\mathcal{E}_I^m$, $\mathcal{E}_M^m$ and $\mathcal{E}_E^m$ that we are interested in. { We conclude that for this kind of scattering data, the mETDRK4 method is still the most efficient since it produces the smallest $L^{\infty}$ error in the shortest computational time.}

The top right subplot in Figure \ref{F:EX2 data} shows that the momentum is conserved for all these five approaches. The bottom subplots track the errors of mass and energy. The results are as expected and similar to the Example 2. 

\begin{figure}
\includegraphics[width=0.45\textwidth]{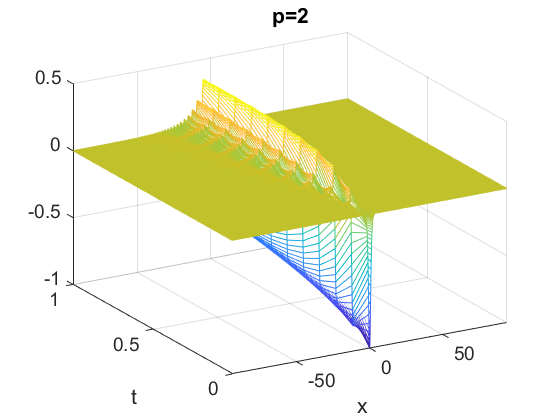}
\includegraphics[width=0.45\textwidth]{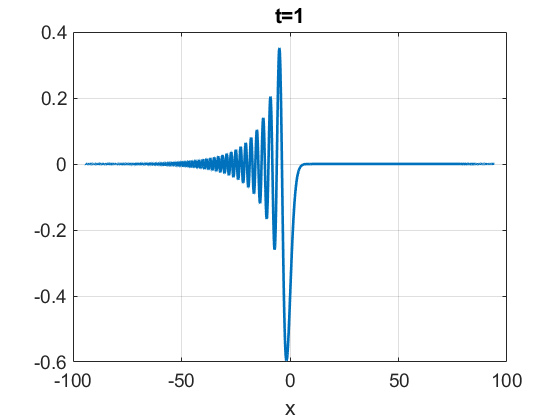}
\caption{\label{F:EX2 profile} The solution profile in Example 3 from SAV-IRK4. Left: $u(x,t)$. Right: $u(x,t)$ at $t=1$.}
\end{figure}

\begin{figure}
\includegraphics[width=0.45\textwidth]{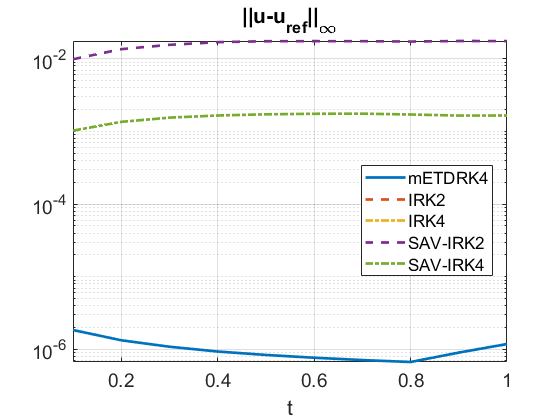}
\includegraphics[width=0.45\textwidth]{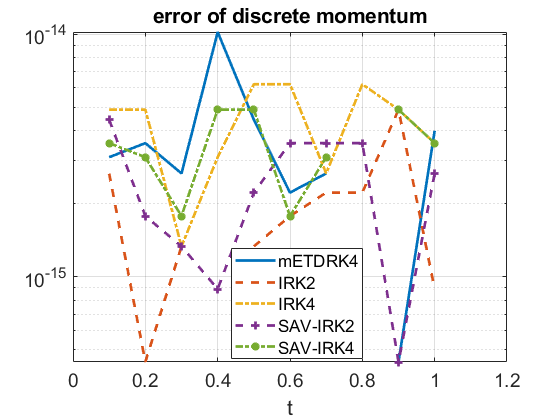}\\
\includegraphics[width=0.45\textwidth]{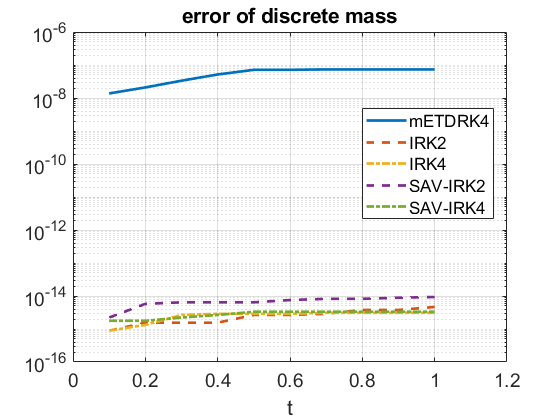}
\includegraphics[width=0.45\textwidth]{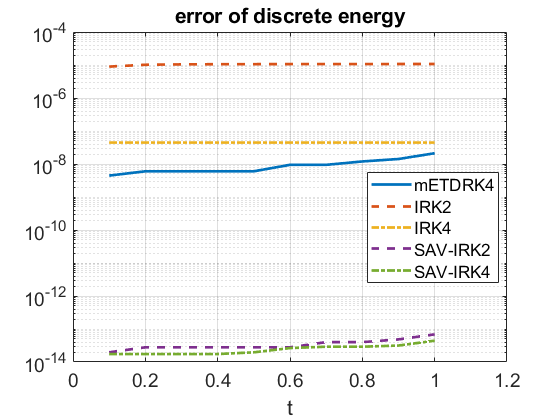}
\caption{\label{F:EX2 data} The errors for Example 3 by different time integrators: mETDEK4 (solid blue); IRK2(dash red, coincides with SAV-IRK2); SAV-IRK2 (dash purple); IRK4 (dash dot orange, coincides with SAV-IRK4); SAV-IRK4 (dash dot green). Top left: $\|u-u_{ref}\|_{\infty}$. Top right: discrete momentum error. Bottom left: discrete mass error. Bottom right: discrete energy error. }
\end{figure}

Table \ref{T:EX2 error1} shows the convergence rates for mETDRK4, SAV-IRK2 and SAV-IRK4. From Table \ref{T:EX2 error1}, we find that the convergence rates is lower than the expected rate for all these approaches. However, as the time step $\tau$ decreases, the rates are approaching their theoretical values ($4$ for the 2nd order methods and $16$ for the 4th order methods). The sub-convergence rate is most likely caused by the fast oscillations of the solution as well as insufficiently small $\tau$.


\begin{center}
\begin{table}
\begin{tabular}{|c|c|c|c|c|c|c|}
\hline
 & \multicolumn{2}{|c|}{mETDRK4} &\multicolumn{2}{|c|}{SAV-IRK2} & \multicolumn{2}{|c|}{SAV-IRK4}\\
\hline
  $\tau$& error & rate &   error  & rate & error  & rate\\
  \hline
 $\frac{1}{100}$& $1.85e-6$ & NA & $1.76e-2$ & NA &$1.76e-3$ & NA  \\
 \hline
 $\frac{1}{200}$& $1.34e-7$ & $13.81$  &$6.38e-3$& $2.76$&$2.85e-4$ &  $6.18$ \\
 \hline
 $\frac{1}{400}$& $1.12e-8$ & $11.96$  &$2.09e-3$ & $3.04$ &$3.74e-5$ & $7.63$ \\
 \hline
 $\frac{1}{800}$& $7.74e-10$ & $14.45$ &$6.01e-4$ & $3.49$ &$3.63e-6$ & $10.28$  \\
 \hline
\end{tabular}
\caption{\label{T:EX2 error1} The convergence rates of mETDRK4, SAV-IRK2 and SAV-IRK4 in Example 3.}
\end{table}
\end{center}

\section{Conclusion}
We propose a numerical scheme for the generalized KdV equations, which can preserve all the three invariant quantities in the discrete time flow with an arbitrarily high order of temporal accuracy { from the the combination of the symplectic Runge-Kutta method and the scalar auxiliary variable reformulation. In fact, instead of the power nonlinearity $\frac{1}{p} u^p$, the general flux $f(u)$ can also be adapted to this algorithm, i.e., one can create the auxiliary variable $v=\sqrt{\mathcal{F}+C_0}$ such that $f=\frac{\delta \mathcal{F}}{\delta u}$. Then, the modified energy ${E}=\frac{1}{2}(u_{x},u_x)-v^2-C_0 $ will be conserved in the discrete SRK time flow.} 

In our numerical experiments, besides the fulfillment of conservation laws for our proposed numerical methods, {we also show that among the popular existing time integrators, the mETDRK4 method is likely to be the most efficient one for the soliton type and the scattering type of solutions. However, for the long time simulation on the oscillatory non-dispersing solutions, such as the breathers for the mKdV equations, the conservative methods are recommended. Our numerical simulations demonstrate the high accuracy and efficiency of the proposed schemes in these cases. 
}


\bibliographystyle{abbrv}
\bibliography{ref_KdVc}
\end{document}